\documentclass{article}
\usepackage[utf8]{inputenc}
\usepackage{comment}
\usepackage{amssymb,amsmath,amsthm,graphicx,xcolor,mathtools,enumerate,mathrsfs}
\usepackage{enumitem,thmtools}
\usepackage{thm-restate}
\usepackage{hyperref}
\usepackage{cleveref}
\usepackage{indentfirst}
\usepackage[left=4cm, right=4cm, top=4cm, bottom=2.25cm]{geometry}

\theoremstyle{plain}
\newtheorem{theorem}{Theorem}[section]
\crefname{theorem}{Theorem}{Theorems}

\crefname{proposition}{Proposition}{Propositions}

\newtheorem{corollary}[theorem]{Corollary}
\crefname{corollary}{Corollary}{Corollaries}

\newtheorem{lemma}[theorem]{Lemma}
\crefname{lemma}{Lemma}{Lemmas}

\crefname{conjecture}{Conjecture}{Conjectures}

\crefname{problem}{Problem}{Problem}

\newtheorem{claim}[theorem]{Claim}
\crefname{claim}{Claim}{Claims}

\crefname{observation}{Observation}{Observations}

\crefname{setup}{Setup}{Setups}

\crefname{fact}{Fact}{Facts}

\crefname{algorithm}{Algorithm}{Algorithms}

\crefname{remark}{Remark}{Remarks}

\crefname{example}{Example}{Examples}

\theoremstyle{definition}

\crefname{definition}{Definition}{Definitions}

\crefname{construction}{Construction}{Constructions}

\crefname{question}{Question}{Questions}

\numberwithin{equation}{section}

\def\eps{\varepsilon}
\renewcommand{\int}[1]{\mathop{\mkern 0mu\mathrm{int}}\nolimits(#1)}

\newcommand\ceil[1]{\left\lceil#1\right\rceil}

\newcommand\cT{\mathcal{T}}
\allowdisplaybreaks

  \makeatletter
  \newcommand{\labelinthm}[1]{%
     \label{temp#1}
     \protected@write \@auxout {}{\string \newlabel{#1}{{\emph{\ref{temp#1}}}{\thepage}{\emph{\ref{temp#1}}}{temp#1}{}} }%
  }
  \makeatother

\title{Spanning trees with large maximum degrees}

\begin{document}
\author{Jun Yan\thanks{Mathematics Institute, University of Oxford, Oxford OX2 6GG, UK. Email: \url{jun.yan@maths.ox.ac.uk}. Supported by
ERC Advanced Grant 883810.}}

\maketitle

\begin{abstract}
The celebrated result of Koml\'os, S\'ark\"ozy, and Szemer\'edi states that for any $\eps>0$, there exists $0<c<1$, such that for all sufficiently large $n$, every $n$-vertex graph $G$ with $\delta(G)\geq(1/2+\eps)n$ contains every $n$-vertex tree with maximum degree at most $cn/\log n$. This is best possible up to the value of $c$. In this paper, we extend this result to trees with higher maximum degrees, and prove that for $\Delta\gg n/\log n$, roughly speaking, $\delta(G)\geq n-n^{1-(1+o(1))\Delta/n}$ is the asymptotically optimal minimum degree condition which guarantees that $G$ contains every $n$-vertex spanning tree with maximum degree at most $\Delta$. We also prove the corresponding statements in the random graph setting. 
\end{abstract}

\section{Introduction}
An important family of problems in extremal graph theory, known as Dirac-type problems, asks for the optimal minimum degree conditions one needs to impose on an $n$-vertex host graph $G$ in order to guarantee that $G$ contains some desired substructures. This is named after the famous Dirac's Theorem~{\cite[Theorem 3]{D}} from 1952, which says that for $n\geq3$, $\delta(G)\geq n/2$ ensures the existence of a Hamilton cycle. Since then, Dirac-type problems have been studied extensively for many different structures. For more information on these and other related problems, we refer interested readers to the survey of K\"uhn and Osthus~\cite{KO}. 

Here, we focus on the case of finding spanning trees with $n$ vertices. For any $n$ and $\Delta$, let $\cT(n,\Delta)$ be the set of $n$-vertex trees with maximum degrees at most $\Delta$. What is the optimal minimum degree condition on an $n$-vertex host graph $G$ that ensures $G$ contains a copy of every $T\in\cT(n,\Delta)$? 

The first result of this type was proved in 1995 by Koml\'os, S\'ark\"ozy, and Szemer\'edi~{\cite[Theorem 1.1]{KSS1}}. They showed that for any constants $\eps,\Delta>0$, and for all sufficiently large $n$, every $n$-vertex graph $G$ with $\delta(G)\geq(1/2+\eps)n$ contains every tree in $\cT(n,\Delta)$. The minimum degree condition was improved in 2010 by Csaba, Levitt, Nagy-Gy\"orgy, and Szemer\'edi~{\cite[Theorem 1]{CLNGS}}, who showed that $\delta(G)\geq n/2+C\log n$ is sufficient for some constant $C$ depending on $\Delta$, and that this is tight up to the value of $C$. 

In 2001, Koml\'os, S\'ark\"ozy, and Szemer\'edi~{\cite[Theorem 1.2]{KSS2}} improved their earlier result in another direction by showing that for any $\eps>0$, there exists $0<c<1$, such that for all sufficiently large $n$, every $n$-vertex graph $G$ with $\delta(G)\geq(1/2+\eps)n$ contains every tree in $\cT(n,cn/\log n)$. Moreover, they showed that this is optimal up to the value of $c$ with a random graph construction. 

In this paper, we consider the natural question of what happens when $\Delta\gg n/\log n$. By Koml\'os, S\'ark\"ozy, and Szemer\'edi's counterexample~\cite{KSS2}, we need to increase the minimum degree condition in order to find all trees in $\cT(n,\Delta)$. Our main result below shows that in this range, 
if an $n$-vertex graph $G$ has minimum degree roughly $n-n^{1-(1+o(1))\Delta/n}$, then $G$ contains every tree in $\cT(n,\Delta)$. 

\begin{theorem}\label{thm:main}
For any $\eps>0$, there exists $C>1$, such that the following is true for all sufficiently large $n$. Let $\Delta\geq Cn/\log n$, and suppose $k=\ceil{(n-1)/\Delta}-1\geq2$. If $G$ is an $n$-vertex graph with $\delta(G)\geq n-n^{1-(1+\eps)/k}$, then $G$ contains every tree in $\cT(n,\Delta)$.
\end{theorem}

Note that when $\Delta=Cn/\log n$, $k$ is roughly $\log n/C$, and the minimum degree condition is around $\delta(G)\geq n-n/e^C$, which is linear. Thus, for $\Delta$ in the lower range, our result asymptotically matches the form of Koml\'os, S\'ark\"ozy, and Szemer\'edi's result in~\cite{KSS2}. 



By a standard application of the Chernoff Bound, we immediately obtain a corresponding random graph version of this result. This deals with a different regime compared to Montgomery's seminal result~{\cite[Theorem 1.1]{M}} in 2019, which shows that for every constant $\Delta$, there exists a constant $C$, such that with high probability $G(n,C\log n/n)$ contains every tree in $\cT(n,\Delta)$. 
\begin{corollary}\label{cor:random}
For any $\eps>0$, there exists $C>1$, such that the following is true. Let $\Delta\geq Cn/\log n$, and suppose $k=\ceil{(n-1)/\Delta}-1\geq2$. Then, with high probability $G(n,1-n^{-(1+2\eps)/k})$ contains every tree in $\cT(n,\Delta)$. 
\end{corollary}
\begin{proof}
For every $v\in G(n,1-n^{-(1+2\eps)/k})$, let $X_v$ be the number of vertices not adjacent to $v$. Then, $\mathbb{E}[X_v]=1+(n-1)n^{-(1+2\eps)/k}$, so by Lemma~\ref{lemma:chernoff1}, for sufficiently large $C$,
\[\mathbb{P}(X_v\geq n^{1-(1+\eps)/k})\leq\mathbb{P}(X_v\geq2\mathbb{E}[X_v])\leq2\exp(-\mathbb{E}[X_v]/3)\leq2\exp(-n^{1-(1+2\eps)/k}/6)\leq o(n^{-1}).\]
Therefore, by a union bound, with high probability $G(n,1-n^{-(1+2\eps)/k})$ has minimum degree at least $n-n^{1-(1+\eps)/k}$, and thus contains every $T\in\cT(n,\Delta)$ by Theorem~\ref{thm:main}.
\end{proof}

When $n/\Delta=O_\eps(1)$ and $n-\Delta\ceil{(n-1)/\Delta}$ is not ``too small", the choice of $k$ in the two results above can be improved by 1 to $\ceil{(n-1)/\Delta}$. By being more careful in our proof, it is possible to make this condition more precise, but we did not pursue this direction in order to present more unified versions of our results that work for all $n$ and $\Delta$. 

Aside from this, both of these results are asymptotically best possible, as shown by the following random graph construction analogous to the counterexample of Koml\'os, S\'ark\"ozy, and Szemer\'edi in~\cite{KSS2}. Recall that a set $K\subset V(G)$ is dominating if $N(K)=V(G)\setminus K$. Essentially, for every $\Delta\gg n/\log n$, the main obstruction to containing all trees in $\cT(n,\Delta)$ is the lack of a dominating set of size around $n/\Delta$, and $p=1-n^{-\Delta/n}$ is around when we can expect $G(n,p)$ to contain such a dominating set.


\begin{theorem}\label{thm:random}
For every $\eps>0$, let $n/\log n\leq\Delta\leq n-1$, and let $k=\ceil{(n-1)/\Delta}$. Then, there exists $T\in\cT(n,\Delta)$, such that with high probability $G(n,1-n^{-(1-\eps)/k})$ does not contain $T$.
\end{theorem}
\begin{proof}
Let $T\in\cT(n,\Delta)$ be a tree that contains a vertex that is adjacent to a set $X$ of $k$ other vertices, such that every other vertex in $T$ is a leaf that is adjacent to a vertex in $X$, and every $x\in X$ has degree at most $\Delta$. This is possible as $k=\ceil{(n-1)/\Delta}$. Moreover, note that $X$ is a dominating set of size $k$ in $T$.

We claim that if $G\sim G(n,1-n^{-(1-\eps)/k})$, then with high probability $G$ does not contain a dominating set $K$ of size $k$, and so with high probability $G(n,1-n^{-(1-\eps)/k})$ does not contain $T$. 
Indeed, for any fixed $K\subset V(G)$, the probability that $K$ is a dominating set in $G$ is at most
\[\left(1-\left(n^{-\frac{1-\eps}k}\right)^k\right)^{n-k}\leq\left(1-n^{-1+\eps}\right)^{n/2}\leq \exp\left(-n^{\eps}/2\right).\]
Therefore, by a union bound, the probability that such a $K$ exists is at most
\[\binom{n}{k}\exp(-n^{\eps}/2)\leq\exp\left(k\log n-n^\eps/2\right)\leq\exp\left(2(\log n)^2-n^\eps/2\right)=o(1),\]
where we used that $k\leq2n/\Delta\leq2\log n$.
\end{proof}


\begin{corollary}\label{cor:notcontain}
For every $\eps>0$, let $n/\log n\leq\Delta\leq n-1$, and let $k=\ceil{(n-1)/\Delta}$. Then, there exists $T\in\cT(n,\Delta)$ and an $n$-vertex graph $G$ with $\delta(G)\geq n-n^{1-(1-2\eps)/k}$, such that $G$ does not contain $T$.
\end{corollary}
\begin{proof}
Let $T$ be as in the proof of Theorem~\ref{thm:random}, so that with high probability, $G(n,1-n^{-(1-\eps)/k})$ does not contain $T$. Using Lemma~\ref{lemma:chernoff1} and as in the proof of Corollary~\ref{cor:random}, with high probability $G(n,1-n^{-(1-\eps)/k})$ has minimum degree $n-n^{1-(1-2\eps)/k}$. Thus, we can find a realisation of $G\sim G(n,1-n^{-(1-\eps)/k})$, such that $\delta(G)\geq n-n^{1-(1-2\eps)/k}$ and $G$ contains no copy of $T$. 
\end{proof}

The rest of the paper is organised as follows. In Section~\ref{sec:prelim}, we first gather several useful preliminary results that will be used later in the main proofs, then provide a brief sketch of the proof of Theorem~\ref{thm:main}. In Section~\ref{sec:main}, we prove Theorem~\ref{thm:main}, by separately proving it in two overlapping ranges $\Delta\geq2n\log\log n/\log n$ (Theorem~\ref{thm:highregime}) and  $Cn/\log n\leq\Delta\ll n$ (Theorem~\ref{thm:lowregime}), in Section~\ref{sec:main1} and Section~\ref{sec:main2}, respectively.

\section{Preliminaries}\label{sec:prelim}
In this section, we first collect several useful preliminary results in Section~\ref{sec:useful}, then outline the proof of Theorem~\ref{thm:main} in Section~\ref{sec:outline}. 
\subsection{Useful results}\label{sec:useful}
We begin by recalling the well-known Azuma's Inequality and Chernoff Bound.
\begin{lemma}[Azuma's Inequality~{\cite[Lemma 4.2]{W}}]\label{lemma:azuma}
Let $X_1,\ldots, X_n$ be a sequence of random variables, such that for each $i\in[n]$, there exist constants $a_i\in\mathbb{R}$ and $c_i>0$ with $|X_i-a_i|\leq c_i$.
\begin{itemize}
    \item If $\mathbb{E}[X_i\mid X_1,\ldots,X_{i-1}]\geq a_i$ for every $i\in[n]$, then for every $t>0$, \[\textstyle\mathbb{P}\left(\sum_{i=1}^n(X_i-a_i)\leq-t\right)\leq\exp\left(-\frac{t^2}{2\sum_{i=1}^nc_i^2}\right).\]
    \item If $\mathbb{E}[X_i\mid X_1,\ldots,X_{i-1}]\leq a_i$ for every $i\in[n]$, then for every $t>0$, \[\textstyle \mathbb{P}\left(\sum_{i=1}^n(X_i-a_i)\geq t\right)\leq\exp\left(-\frac{t^2}{2\sum_{i=1}^nc_i^2}\right).\]
\end{itemize}
\end{lemma}

\begin{lemma}[Chernoff Bound~{\cite[Corollary 2.3, Theorem 2.10]{JLR}}]\label{lemma:chernoff1}
Let $X$ be either a binomial random variable or a hypergeometric random variable. Then, for every $0<\eps\le 3/2$,
\[\mathbb{P}\left(|X-\mathbb E[X]|\ge \eps\mathbb E[X]\right)\le 2\exp(-\eps^2\mathbb{E}[X]/3).\]
\end{lemma}

We also use another version of the Chernoff Bound. Let $D(x\parallel y)$ denote the quantity $x\log(x/y)+(1-x)\log((1-x)/(1-y))$. In this paper, we only use this notation as a shorthand. However, we remark that more generally, given two probability distributions $P$ and $Q$, $D(P\parallel Q)$ denotes the Kullback--Leibler divergence, which is a measure of how different $Q$ is from $P$. In particular, the quantity $D(x\parallel y)$ that we are interested in here is the Kullback--Leibler divergence between two Bernoulli random variables with parameters $x$ and $y$, respectively.
\begin{lemma}[Chernoff Bound]\label{lemma:chernoff2}
Let $X$ be either the binomial random variable $\textup{Bin}(n,q)$ or the hypergeometric random variable $\textup{HG}(N,K,n)$ with $K/N=q$. Then, for every $0<\lambda<1-q$,
\[\mathbb{P}(X\geq(q+\lambda)n)\leq\exp\left(-nD(q+\lambda\parallel q)\right).\]
In particular, \[\mathbb{P}(X\geq(q+\lambda)n)\leq(2q^\lambda)^n.\]
\end{lemma}
\begin{proof}
The first part of the lemma follows directly from~{\cite[Theorem 1, Theorem 4]{Ho}}. Now we prove the second part. It is easy to check that the function $f(t)=(t-\lambda)\log((t-\lambda)/t)$ is decreasing on $(\lambda,1]$. Therefore, from definition, 
\begin{align*}
D(q+\lambda\parallel q)&=(q+\lambda)\log((q+\lambda)/q)+(1-q-\lambda)\log((1-q-\lambda)/(1-q))\\
&\geq\lambda\log(1/q)+\lambda\log\lambda+(1-\lambda)\log(1-\lambda)\geq\lambda\log(1/q)-\log2,
\end{align*}
from which it follows that $\mathbb{P}(X\geq(q+\lambda)n)\leq\exp\left(-n\lambda\log(1/q)+n\log2\right)=(2q^{\lambda})^n$.
\end{proof}

In many of our later tree embedding arguments, we will first embed all but a small set of vertices with degrees 1 or 2 in $T$. To embed these low degree vertices at the end, we use a Hall type matching argument via the following well-known Hall's Theorem and its generalisation. The conditions in Lemma~\ref{lemma:Hall} and Lemma~\ref{lemma:hallmatching} will both be referred to as Hall's condition.
\begin{lemma}[Hall's Theorem {\cite[Theorem 1]{H}}]\label{lemma:Hall}
Let $G$ be a bipartite graph with bipartition classes $A$ and $B$. If $|N(S)|\ge |S|$ for every $S\subset A$, then $G$ contains a matching covering all vertices in $A$.
\end{lemma}
\begin{lemma}[{\cite[Corollary 11]{B}}]\label{lemma:hallmatching}
Let $G$ be a bipartite graph with bipartition classes $A$ and $B$, and let $(f_a)_{a\in A}$ be a tuple of non-negative integers indexed by elements of $A$. Suppose that $|N(S)|\geq\sum_{a\in S}f_a$ for every $S\subset A$. Then, there exists a collection of vertex-disjoint stars $(S_a)_{a\in A}$ in $G$, such that for each $a\in A$, $S_a$ is centred at $a$ and has exactly $f_a$ leaves.
\end{lemma}

Finally, we record the following useful lemma, which says that every tree either has many bare paths or many leaves. In what follows, a path contained in a tree is called a bare path if all of its internal vertices have degree exactly 2. 
\begin{lemma}[{\cite[Lemma 2.1]{K}}]\label{lemma:paths-leaf}
Let $k,\ell, n$ be positive integers. Then, every $n$-vertex tree $T$ either contains at least $\ell$ leaves, or contains a collection of at least $\frac{n}{k+1}-(2\ell-2)$ vertex-disjoint bare paths of length $k$.
\end{lemma}

\subsection{Proof overview}\label{sec:outline}
In this subsection, we give an overview of the proof of Theorem~\ref{thm:main}. We begin by applying Lemma~\ref{lemma:paths-leaf} with suitable parameters to conclude that $T$ either contains many vertex-disjoint bare paths of length 4, or many leaves. 

If $T$ contains a collection $\mathcal{P}$ of many bare paths of length 4, this is dealt with easily by Lemma~\ref{lemma:bare-paths}, which is adapted from{~\cite[Lemma 6.1]{MPY}}. We first greedily embed into $G$ all vertices in $T$, except the three central vertices on each bare path in $\mathcal{P}$, using the high minimum degree condition, then use a Hall type matching argument to embed the central vertices. This is proved in Section~\ref{sec:barepath}. 

Now assume that $T$ contains many leaves, we consider two cases. Here, the motivation is Lemma~\ref{lemma:manyleaves}, which is proved in Section~\ref{sec:manyleaves}. It says that if $T$ contains a set $L$ of many leaves, such that every parent of leaves in $L$ is adjacent to at most $O(|L|/\log n)$ leaves in $L$, then we can embed $T$ into a graph with high minimum degree. The proof is again adapted from similar results in~\cite{MPY}. We first embed $T-L$ randomly, then use the randomness to ensure that every remaining vertex in the graph is a good candidate for a positive fraction of leaves in $L$ to be embedded there. This condition is enough for us to apply a Hall type matching argument to embed all leaves in $L$. The $\log n$ factor in the parent degree condition mentioned above is crucial, as it ensures the concentration step will work out.

It will turn out that the case division is based on whether $T$ contains around $k$ parents with very high leaf degrees. If not, then the leaves are spread out enough among their parents that we can apply Lemma~\ref{lemma:manyleaves} to embed $T$. Otherwise, the embedding of these around $k$ parents with high leaf degrees is the main bottleneck. The most extreme example is exactly the one used to prove Theorem~\ref{thm:random}, where $T$ contains a dominating set of size $k$, and each vertex inside is adjacent to around $n/k\approx\Delta$ leaves. The condition $\delta(G)\geq n-n^{1-(1+\eps)/k}$ does ensure the existence of a dominating set $K$ of size $k$, into which we should embed those parents, but this alone is not enough to embed the tree $T$. The Hall type matching argument we use at the end to embed the leaves generally requires a stronger property than $K$ being dominating. Moreover, these parents may be adjacent to each other, and many of them could also share a common neighbour, so it is also not straightforward to accurately place them inside $K$. Due to these technicalities, we need to prove Theorem~\ref{thm:main} separately with the following two results that deal with two overlapping ranges of $\Delta$. Note that in Theorem~\ref{thm:lowregime}, we can redefine $k$ to be $n/\Delta$ for convenience as the difference can be absorbed into the $\eps$ factor in this range.
 
\begin{theorem}\label{thm:highregime}
For any $\eps>0$, there exists $C>1$, such that the following is true for all sufficiently large $n$. Let $\Delta\geq2n\log\log n/\log n$, and suppose $k=\ceil{(n-1)/\Delta}-1\geq2$. If $G$ is an $n$-vertex graph with $\delta(G)\geq n-n^{1-(1+\eps)/k}$, then $G$ contains every tree in $\cT(n,\Delta)$.
\end{theorem}

\begin{theorem}\label{thm:lowregime}
For any $\eps>0$, there exists $C>1$, such that the following is true for all sufficiently large $n$. Let $Cn/\log n\leq\Delta\ll n$, and let $k=n/\Delta$. If $G$ is an $n$-vertex graph with $\delta(G)\geq n-n^{1-(1+\eps)/k}$, then $G$ contains every tree in $\cT(n,\Delta)$.
\end{theorem}

The high range when $\Delta\gg n\log\log n/\log n$ and $k\ll\log n\log\log n$ is easier, as in this range we can actually find a dominating set $K$ of size $k$ with $G[K]$ being a clique. Moreover, since $kn^{1-(1+\eps)/k}\ll n$ in this range, as long as $|L|>kn^{1-(1+\eps)/k}$, while embedding $T-L$, we can easily find a common neighbour for any $k$ vertices. This allows us to embed $T-L$ mostly greedily, while precisely placing the $k$ parents with large leaf degrees into $K$. Finally, since $\Delta\gg n^{1-(1+\eps)/k}$ in this range, $K$ being dominating turns out to suffice for embedding $L$ with a simple Hall type argument. This range is handled in Section~\ref{sec:main1}.

In the low range when $Cn/\log n\leq\Delta\ll n$, in particular in the low end when $\Delta=\Theta(n/\log n)$, $kn^{1-(1+\eps)/k}=\Theta(n\log n)$, so a mostly greedy embedding will not work because we cannot guarantee that any $k$ vertices have a common neighbour. Furthermore, $K$ being dominating is no longer enough for the Hall type argument at the end. To combat these, we will instead use a randomised approach, which requires $k\gg1$ and so $\Delta\ll n$. Let $R_1$ and $R_2$ be disjoint random subsets of $V(G)$ with sizes roughly $k$ and $n^{\eps/100}$, respectively. Then, we can show that with high probability, every vertex is adjacent to at least half of the vertices in $R_2$, and every vertex in $R_1\cup R_2$ is adjacent to at least a $(1-\eps/20)$-fraction of the vertices in $R_1$. These two conditions together allow us to embed $T-L$ while ensuring that at least a $(1-\eps/20)$-fraction of the parents are embedded into $R_1$. Moreover, we can show that with high probability, every vertex is adjacent to an $(\eps/10)$-fraction of vertices in $R_1$, so that the Hall type argument works to embed the leaves. This range is handled in Section~\ref{sec:main2}.

\section{Proof of Theorem~\ref{thm:main}}\label{sec:main}
In this section, we prove our main result, Theorem~\ref{thm:main}. As mentioned in Section~\ref{sec:outline}, it suffices to prove Theorem~\ref{thm:highregime} and~\ref{thm:lowregime}, which we will do in Section~\ref{sec:main1} and Section~\ref{sec:main2}, respectively. Before that, we first prove two intermediate results in Section~\ref{sec:barepath} and Section~\ref{sec:manyleaves}, which show that a tree $T$ can be embedded into a graph with high minimum degree if $T$ has many bare paths, and if $T$ has many leaves sufficiently spread out among their parents, respectively.

\subsection{Many bare paths}\label{sec:barepath}
In this subsection, we prove the following lemma, which says that if a tree contains many vertex-disjoint bare paths of length 4, then it can be embedded into a graph with high minimum degree. This generalises~{\cite[Lemma 6.1]{MPY}}, though the proof is essentially the same.
\begin{lemma}\label{lemma:bare-paths}
Let $m\leq n/100$, let $G$ be an $n$-vertex graph with $\delta(G)\ge n-m$, and let $T$ be an $n$-vertex tree that contains $10m$ vertex-disjoint bare paths of length 4. Then, $G$ contains a copy of $T$.
\end{lemma}
\begin{proof}
From assumption, there is a collection $\mathcal P=\{P_1,\ldots,P_\ell\}$ of $\ell=10m$ vertex-disjoint bare paths in $T$ of length $4$. Let $T'$ be the forest obtained by removing all internal vertices of the paths in $\mathcal P$ from $T$, so that $|T'|=n-3\ell$.

Since $\delta(G)\geq n-m\geq |T'|$, we can greedily find a copy $S'$ of $T'$ in $G$. Let $G'=G-V(S')$. Note that $|G'|=3\ell=30m$, so $\delta(G')\geq|G'|-m\geq|G'|/2$. Thus, by Dirac's Theorem, $G'$ contains a Hamilton cycle. In particular, we can label the $3\ell$ vertices in $G'$ as $w_1,\ldots,w_\ell,x_1,\ldots,x_\ell,y_1,\ldots,y_\ell$, so that for each $i\in [\ell]$, $x_iw_iy_i$ is a path in $G'$.

For each $i\in [\ell]$, let $u_i$ and $v_i$ be the copies of the two endpoints of $P_i$ in $S'$. Let $K$ be an auxiliary bipartite graph with bipartition classes $A=\{a_1,\ldots,a_\ell\}$ and $B=\{b_1,\ldots,b_\ell\}$, such that for any $i,j\in[\ell]$, there is an edge $a_ib_j$ in $K$ if and only if both $u_ix_j$ and $v_iy_j$ are edges in $G$. Since $\delta(G)\geq n-m$, for every $i\in[\ell]$, $d_K(a_i),d_K(b_i)\geq\ell-2m$. Then, for any $I\subset A$ with $0<|I|\leq\ell-2m$, we have $|N_K(I,B)|\geq\ell-2m\geq |I|$. For any $I\subset A$ with $|I|>\ell-2m$, we have $|N_K(I,B)|=|B|\geq|I|$, as any $b\in B\setminus N_K(I,B)$ would satisfy $d_K(b)<2m<\ell-2m$, a contradiction. Thus, by Lemma~\ref{lemma:Hall}, there is a perfect matching in $K$, say matching $a_i$ with $b_{\sigma(i)}$ for every $i\in[\ell]$. This then implies that $S'$ along with the paths $u_ix_{\sigma(i)}w_{\sigma(i)}y_{\sigma(i)}v_i$ for all $i\in[\ell]$ form a copy of $T$ in $G$, as required.
\end{proof}

\subsection{Many leaves and many parents}\label{sec:manyleaves}
In this subsection, we prove the following lemma, which says that if a tree contains a set $L$ of many leaves, such that their parents all have at most $|L|/10\log n$ neighbours in $L$, then it can be embedded into a graph with high minimum degree. This is a simpler version of several related results in~\cite{MPY}.
\begin{lemma}\label{lemma:manyleaves}
Let $m\leq n/20$, let $G$ be an $n$-vertex graph with $\delta(G)\geq n-m$, and let $T$ be an $n$-vertex tree. Let $L$ be a set of leaves in $T$, and let $P$ be the set of parents of the leaves in $L$. Suppose that $|L|\geq 5m$, and $\max\{d(p,L):p\in P\}\leq|L|/10\log n$, then $G$ contains a copy of $T$.  
\end{lemma}
\begin{proof}
Let $P=\{p_1,\ldots,p_\ell\}$, and let $d_i=d(p_i,L)$ for every $i\in[\ell]$. Order the vertices in the tree $T-L$ as $t_1,\ldots,t_r$, such that $t_1=p_1$, and for every $2\leq i\leq r$, $t_i$ has a unique neighbour to its left in this ordering. Moreover, by relabelling if necessary, we may assume that $p_1,\ldots,p_\ell$ appear in this order. Consider the following random embedding $\psi$ of $T-L$ into $G$. Pick $\psi(t_1)\in V(G)$ uniformly at random. Then, for each $2\leq i\leq r$, let $j_i\in[i-1]$ be the unique index such that $t_{j_i}$ is a neighbour of $t_i$, and pick $\psi(t_i)$ uniformly at random from the set $Y_i=N(\psi(t_{j_i}))\setminus\{\psi(t_1),\ldots,\psi(t_{i-1})\}$. This procedure always succeeds in producing a random embedding of $T-L$, as $\delta(G)\geq n-m$ and $|L|\geq5m$ together imply that $|Y_i|\geq4m>0$. 

Now, note that for every $w\in V(G)$ and every $2\leq i\leq\ell$, if $p_i=t_{i'}$ and $Y_{i'}$ is defined as above, then,
\[\mathbb{P}(w\not\in N(\psi(p_i))\mid\psi(t_1),\ldots,\psi(t_{i'-1}))\leq\frac{|Y_{i'}\setminus N(w)|}{|Y_{i'}|}\leq\frac{m}{4m}<\frac12.\]
Also, $\mathbb{P}(w\not\in N(\psi(p_1)))\leq m/n<1/2$. Therefore, by Lemma~\ref{lemma:azuma},
\begin{align*}
\mathbb{P}\left(\sum_{i:w\in N(\psi(p_i))}d_i<|L|/4\right)&\leq\exp\left(-\frac{(|L|/4)^2}{2\frac{|L|/4}{|L|/10\log n}(|L|/10\log n)^2}\right)\\
&=\exp\left(-\frac{5|L|}{4|L|/\log n}\right)=\exp\left(-\frac{5\log n}{4}\right)=o(n^{-1}).
\end{align*}
By a union bound, with high probability, every $w\in V(G)$ satisfies $\sum_{i:w\in N(\psi(p_i))}d_i\geq|L|/4$. Fix an embedding $\psi$ with this property.

Let $W$ be the set of remaining vertices, it suffices now to embed $L$ appropriately into $W$, noting that $|L|=|W|$. To do this, we verify Hall's condition. For any $\varnothing\not=I\subset[\ell]$, if $\sum_{i\in I}d_i\leq|L|-m$, then $|N(\psi(\{p_i:i\in I\}),W)|\geq|W|-m\geq\sum_{i\in I}d_i$ as $\delta(G)\geq n-m$. If $I\subset[\ell]$ satisfies $\sum_{i\in I}d_i>|L|-m$, then $N(\psi(\{p_i:i\in I\}),W)=W$ and $|N(\psi(\{p_i:i\in I\}),W)|=|W|=|L|\geq\sum_{i\in I}d_i$, because any $w\in W\setminus N(\psi(\{p_i:i\in I\}),W)$ would satisfy $\sum_{i:w\in N(\psi(p_i))}d_i\leq m<|L|/4$, a contradiction. Hence, Hall's condition holds and we can use Lemma~\ref{lemma:hallmatching} to finish the embedding. 
\end{proof}

\subsection{High range}\label{sec:main1}
In this subsection, we prove Theorem~\ref{thm:highregime}, which deals with the high range $\Delta\geq2n\log\log n/\log n$.
\begin{proof}[Proof of Theorem~\ref{thm:highregime}]
Since $\Delta\geq2n\log\log n/\log n$, $k=\ceil{(n-1)/\Delta}-1\leq n/\Delta\leq\log n/2\log\log n$, so 
\[n^{1-\frac{1+\eps}{k}}\leq n\exp\left(-\frac{(2+2\eps)\log n\log\log n}{\log n}\right)=\frac{n}{(\log n)^{2+2\eps}}\ll\frac\Delta{\log n}\ll\Delta.\]
Therefore, if $T$ contains $\Delta/100$ vertex-disjoint bare paths of length 4, then $G$ contains a copy of $T$ by Lemma~\ref{lemma:bare-paths}. Assume from now on that this is not the case.  

Let $L$ be the set of leaves in $T$, and let $P$ be the set of parents of $L$. Let $P^+=\{p\in P:d(p,L)>n^{1-(1+\eps)/k}\}$, and let $P^-=P\setminus P^+$. Let $L^+=N(P^+,L)$, and let $L^-=N(P^-,L)$. 

\medskip

\noindent\textbf{Case I.} $|P^+|\leq k-1$. Then, $|T-L^+|\geq n-(k-1)\Delta>\Delta$.

Suppose $T-L^+$ contains $\Delta/50$ vertex-disjoint bare paths of length 4. Note that unless such a bare path contains a vertex in $P^+$, it would also be a bare path in $T$. Hence, $T$ contains at least $\Delta/50-(k-1)\geq\Delta/100$ vertex-disjoint bare paths of length 4, a contradiction. 

Thus, by Lemma~\ref{lemma:paths-leaf}, $T-L^+$ contains at least $\Delta/20$ leaves. Since at most $k-1$ of these are not in $L^-$, we can find a set $L'$ of $\Delta/100$ leaves in $T-L^+$ that are in $L^-$. Let $P'=\{p_1,\ldots,p_\ell\}$ be the set of parents of $L'$. For every $i\in[\ell]$, let $d_i=d(p_i,L')$, and note that $d_i\leq n^{1-(1+\eps)/k}$. Since $\Delta/\log n\gg n^{1-(1+\eps)/k}$, we can apply Lemma~\ref{lemma:manyleaves} to find a copy of $T$ in $G$.

\medskip

\noindent\textbf{Case II.} $|P^+|\geq k$. Arbitrarily pick distinct $p_1,\ldots,p_k\in P^+$, let $P'=\{p_1,\ldots,p_k\}$, and let $L'=N(P',L)$. 

\begin{claim}\label{claim1}
There exists a dominating set $K\subset V(G)$ of size $k$ with $G[K]$ being a clique.
\end{claim}
\begin{proof}[Proof of Claim~\ref{claim1}]
Consider the set $X\subset V(G)^{k+1}$ of all ordered tuples $(v_1,\ldots,v_k,w)$ of distinct vertices in $G$, such that $v_1,\ldots,v_k$ form a clique, and $w$ is adjacent to none of $v_1,\ldots,v_k$. 

Using the minimum degree condition, the number of choices for $v_1,\ldots,v_k$ is at least 
\[\prod_{i=0}^{k-1}\left(n-in^{1-\frac{1+\eps}k}\right)=n^k\prod_{i=0}^{k-1}\left(1-in^{-\frac{1+\eps}k}\right)\geq n^k\exp\left(-\sum_{i=0}^{k-1}2in^{-\frac{1+\eps}k}\right)\geq n^k\exp\left(-k^2n^{-\frac{1+\eps}k}\right).\]
If such a clique $K$ does not exist, then any choice of $v_1,\ldots,v_k$ has at least one corresponding choice of $w$ not adjacent to any $v_1,\ldots,v_k$, so $|X|\geq n^k\exp(-k^2n^{-(1+\eps)/k})$.

On the other hand, for any choice of $w$, there are at most $(n^{1-(1+\eps)/k})^k$ corresponding choices of $v_1,\ldots,v_k$ from the minimum degree condition, so $|X|\leq n\cdot(n^{1-(1+\eps)/k})^k=n^{k-\eps}$. 

However, as $k\leq\log n/2\log\log n$, 
\begin{align*}
|X|\geq n^k\exp\left(-k^2n^{-\frac{1+\eps}k}\right)&\geq n^k\exp\left(-k^2(\log n)^{-2-2\eps}\right)\\
&\geq n^k\exp\left(-\frac1{(\log n)^{2\eps}(\log\log n)^2}\right)>n^k\exp\left(-\eps\log n\right)=n^{k-\eps}\geq|X|,
\end{align*}
which is a contradiction. 
\renewcommand{\qedsymbol}{$\boxdot$}
\end{proof}
\renewcommand{\qedsymbol}{$\square$}
Let $K=\{v_1,\ldots,v_k\}$ be the dominating set given by Claim~\ref{claim1}. Order the vertices of $T-L'$ as $t_1,\ldots,t_r$, such that $t_1=p_1$, and for every $2\leq i\leq r$, $t_i$ has a unique neighbour to its left in this ordering. Embed $p_1$ to $v_1$. Greedily, extend this to an embedding of $T-L'$, with $p_i$ embedded to $v_i$ for every $i\in[k]$. This is possible as $|L'|>kn^{1-(1+\eps)/k}$, and in the worst case we need to ensure that the next chosen vertex is in the common neighbourhood of $k$ vertices. 

Let $W$ be the set of remaining vertices in $G$, it suffices now to embed $L'$ into $W$. Let $d_i=|N(p_i,L')|$ for each $i\in[k]$, we again verify Hall's condition. For every $\varnothing\not=I\subsetneq[k]$, $\sum_{i\in I}d_i\leq |L'|-n^{1-(1+\eps)/k}$ from the definition of $P^+$. Thus, $|N(\{v_i:i\in I\},W)|\geq|L'|-n^{1-(1+\eps)/k}\geq\sum_{i\in I}d_i$ from the minimum degree condition. Also, since $K$ is a dominating set, $|N(\{v_i:i\in[k]\},W)|=|N(K,W)|=|W|=|L'|=\sum_{i\in[k]}d_i$. Therefore, Hall's condition holds, and we can finish the embedding of $T$ using Lemma~\ref{lemma:hallmatching}.
\end{proof}

\subsection{Low range}\label{sec:main2}
In this subsection, we prove Theorem~\ref{thm:lowregime}, which deals with the low range $Cn/\log n\leq\Delta\ll n$.
\begin{proof}[Proof of Theorem~\ref{thm:lowregime}]
Let $C\gg1/\eps$. Since $Cn/\log n\leq\Delta\ll n$, we have $1\ll k\leq\log n/C$, and $n^{1-(1+\eps)/k}\leq n/n^{(1+\eps)C/\log n}=n/e^{(1+\eps)C}$.

Let $e^{-(1+\eps)C}\ll1/C\ll\mu\ll\eps$. If $T$ contains $\mu n/100$ vertex-disjoint bare paths of length $4$, then as $\delta(G)\geq n-n^{1-(1+\eps)/k}\geq n-\mu n/1000$, we can use Lemma~\ref{lemma:bare-paths} to find a copy of $T$ in $G$. Suppose that this is not the case. Set $k'=(1-\mu)k$, and split into two cases. 

\medskip

\noindent\textbf{Case I.} The number of vertices in $T$ that are adjacent to at least $n/C\log n$ leaves is at most $k'$. Let $P^+$ be the set of such vertices in $T$, and let $L^+$ be the set of leaves they are adjacent to. Then, $|T-L^+|\geq n-k'\Delta\geq n-(1-\mu)n=\mu n$. If $T-L^+$ contains at least $\mu n/50$ vertex-disjoint bare paths of length 4, then $T$ contains at least $\mu n/50-k'\geq\mu n/100$ vertex-disjoint bare paths of length 4, a contradiction. Hence, by Lemma~\ref{lemma:paths-leaf}, $T-L^+$ contains at least $\mu n/50$ leaves. Since at most $k'$ of these are in $P^+$, we can find a set $L'$ of $\mu n/100$ leaves in $T-L^+$ that are also leaves in $L$. Let $P'$ be the set of parents of $L'$, and note that every parent in $P'$ is adjacent to at most $n/C\log n$ leaves in $L'$. Since $(\mu n/100)/(n/C\log n)=\mu C\log n/100\geq10\log n$, we can apply Lemma~\ref{lemma:manyleaves} to find a copy of $T$ in $G$. 

\medskip

\noindent\textbf{Case II.} $T$ contains a set $P^+$ of $k'$ vertices, each of which is adjacent to at least $n/C\log n$ leaves in $T$. Let $L^+$ be the set of leaves $P^+$ is adjacent to. Order the vertices in $T-L^+$ as $t_1,\ldots,t_r$, such that $t_1\in P^+$, and for every $2\leq i\leq r$, $t_i$ has a unique neighbour to its left in this ordering. Let $Q$ be the set of vertices that are the unique left neighbours of vertices in $P^+$ in this ordering, so $|Q|\leq k'$. Note that $Q$ and $P^+$ may intersect. 
\begin{claim}\label{claim:ineq}
$\eps k'n/100C\log n\geq n^{1-(1+\eps)/k}$.
\end{claim}
\begin{proof}[Proof of Claim~\ref{claim:ineq}]
Since $k'=(1-\mu)k\geq k/2$, it suffices to show that $\eps kn^{(1+\eps)/k}\geq 200C\log n$. Consider the function $f(t)=tn^{(1+\eps)/t}$, whose derivative is 
\[f'(t)=n^{(1+\eps)/t}-(1+\eps)n^{(1+\eps)/t}\log n/t=n^{(1+\eps)/t}(1-(1+\eps)\log n/t).\]
Since $k\leq\log n/C$, $f(k)\geq f(\log n/C)=n^{(1+\eps)C/\log n}\log n/C\geq e^{C}\log n/C$. Hence, $\eps kn^{(1+\eps)/k}\geq\eps e^C\log n/C\geq 200C\log n$.
\renewcommand{\qedsymbol}{$\boxdot$}
\end{proof}
\renewcommand{\qedsymbol}{$\square$}

In particular, Claim~\ref{claim:ineq} implies that $|L^+|\geq k'n/C\log n\gg n^{1-(1+\eps)/k}$.

Let $R_1\subset V(G)$ be a uniformly random subset of size $k'$, and let $R_2\subset V(G)\setminus R_1$ be a uniformly random subset of size $n^{\eps/100}$. For every $v\in V(G)$, let $X_v$ and $Y_v$ be the number of non-neighbours of $v$ in $R_1$ and $R_2$, respectively. Then, $X_v$ is the hypergeometric random variable $\textup{HG}(n,m,k')$ with $m\leq n^{1-(1+\eps)/k}$ and $q=m/n\leq n^{-(1+\eps)/k}$. Thus, by Lemma~\ref{lemma:chernoff2}, for every $0<\lambda<1-q$,
\begin{align*}
\mathbb{P}(X_v>(q+\lambda)k')&\leq(2q^\lambda)^{k'}\leq(2n^{-(1+\eps)\lambda/k})^{(1-\mu)k}\\
&=n^{-(1+\eps)(1-\mu)\lambda+(1-\mu)k\log 2/\log n}\leq n^{-(1+\eps/2)\lambda+\log 2/C},
\end{align*}
where in the last step we used $\mu\ll\eps$ and $k\leq\log n/C$.

First, apply this with $\lambda=1-\eps/4$. By a union bound, with probability $1-o(1)$, every $v\in V(G)$ satisfies $X_v\leq(n^{-(1+\eps)/k}+1-\eps/4)k'\leq(1-\eps/10)k'$. Equivalently, every $v\in V(G)$ is adjacent to at least $\eps k'/10$ vertices in $R_1$. Then, apply this again with $\lambda=\eps/30$. By a union bound, with probability $1-(k'+n^{\eps/100})\cdot n^{-\eps/50}=1-o(1)$, every $v\in R_1\cup R_2$ is adjacent to at least $(1-\eps/20)k'$ vertices in $R_1$. 

Similarly, $Y_v$ is the hypergeometric random variable $\textup{HG}(n,m,n^{\eps/100})$ with $q=m/n\leq n^{-(1+\eps)/k}$. By Lemma~\ref{lemma:chernoff2} applied with $\lambda=1/3$,
\[\mathbb{P}(Y_v>n^{\eps/100}/2)\leq(2n^{-(1+\eps)/3k})^{n^{\eps/100}}=o(n^{-1}).\] 
Thus, with probability $1-o(1)$, every $v\in V(G)$ is adjacent to at least half of the vertices in $R_2$. Fix realisations of $R_1$ and $R_2$ with all of the above properties.

Now, consider the following procedure that finds an embedding $\psi$ of $T-L^+$ into $G$ by embedding $t_1,\ldots,t_r$ in this order. Recall that $t_1\in P^+$, arbitrarily pick $\psi(t_1)\in R_1$. For every $2\leq i\leq r$, let $j_i\in[i-1]$ be the unique index such that $t_{j_i}$ is adjacent to $t_i$, and do the following. 
\begin{enumerate}[label=\alph*)]
    \item\label{1} If $t_i\in P^+$, embed it into $N(\psi(t_{j_i}),R_1)\setminus\{\psi(t_1),\ldots,\psi(t_{i-1})\}$ if possible, and into $N(\psi(t_{j_i}))\setminus(\{\psi(t_1),\ldots,\psi(t_{i-1})\}\cup R_1\cup R_2)$ otherwise.
    \item\label{2} If $t_i\in Q\setminus P^+$, embed it into $N(\psi(t_{j_i}),R_2)\setminus\{\psi(t_1),\ldots,\psi(t_{i-1})\}$.
    \item\label{3} If $t_i\not\in Q\cup P^+$, embed it into $N(\psi(t_{j_i}))\setminus(\{\psi(t_1),\ldots,\psi(t_{i-1})\}\cup R_1\cup R_2)$. 
\end{enumerate}

Note that~\ref{1} and~\ref{3} are always possible as $\delta(G)\geq n-n^{1-(1+\eps)/k}$, and $|L^+|\geq k'n/C\log n\geq|R_1|+|R_2|+n^{1-(1+\eps)/k}$. \ref{2} is always possible as $\psi(t_{j_i})$ has at least $|R_2|/2>|Q|$ neighbours in $R_2$ using the defining property of $R_2$, and only vertices in $Q$ are embedded into $R_2$. Finally, we claim that $\psi$ embeds at least $(1-\eps/20)k'$ vertices in $P^+$ into $R_1$. Indeed, only vertices in $P^+$ can be embedded into $R_1$. As long as fewer than $(1-\eps/20)k'$ vertices in $P^+$ have been embedded into $R_1$, $\psi(t_{j_i})\in R_1\cup R_2$, so $\psi(t_{j_i})$ is adjacent to at least $(1-\eps/20)k'$ vertices in $R_1$ by the defining properties of $R_1$ and $R_2$, and~\ref{1} can be carried out with $t_i$ embedded into $R_1$.

Let $P'$ be the set of vertices in $P^+$ embedded into $R_1$ by $\psi$, and let $L'=N(P',L^+)$. Greedily embed every leaf in $L^+\setminus L'$, which is possible as $|L'|\geq(1-\eps/20)k'n/C\log n\gg n^{1-(1+\eps)/k}$.

Let $W$ be the set of remaining vertices in $G$, and note that $|L'|=|W|$. It remains to embed $L'$ into $W$. Again, we verify Hall's condition. Let $P'=\{p_1,\ldots,p_\ell\}$, where $\ell\geq(1-\eps/20)k'$. For every $i\in[\ell]$, let $d_i=d(p_i,L')$, and note that $d_i\geq n/C\log n$. For every $\varnothing\not=I\subset [\ell]$ with $\sum_{i\in I}d_i\leq|L'|-n^{1-(1+\eps)/k}$, $N(\psi(\{p_i:i\in I\}),W)\geq|W|-n^{1-(1+\eps)/k}\geq\sum_{i\in I}d_i$ by the minimum degree condition. Now suppose that $\sum_{i\in I}d_i>|L'|-n^{1-(1+\eps)/k}$. Then, $|[\ell]\setminus I|\leq Cn^{1-(1+\eps)/k}\log n/n\leq\eps k'/100$ by Claim~\ref{claim:ineq}, so $|I|\geq(1-\eps/15)k'$. Since every $w\in W$ is adjacent to at least $\eps k'/10$ vertices in $R_1$, it must be adjacent to some vertex in $\psi(\{p_i:i\in I\})$, so $|N(\psi(\{p_i:i\in I\}),W)|=|W|\geq\sum_{i\in I}d_i$. Therefore, Hall's condition holds, and we can apply Lemma~\ref{lemma:hallmatching} to finish the embedding.  
\end{proof}

\bibliographystyle{plain}
\bibliography{bibliography}

\begin{thebibliography}{10}

\bibitem{B}
B{\'e}la Bollob{\'a}s.
\newblock {\em Modern Graph Theory}.
\newblock Springer, 1998.

\bibitem{CLNGS}
B\'ela Csaba, Ian Levitt, Judit Nagy-Gy\"orgy, and Endre Szemeredi.
\newblock Tight bounds for embedding bounded degree trees.
\newblock In {\em Fete of {C}ombinatorics and {C}omputer {S}cience}, pages 95--137. Springer, 2010.

\bibitem{D}
Gabriel~A. Dirac.
\newblock Some theorems on abstract graphs.
\newblock {\em Proceedings of the London Mathematical Society}, s3-2(1):69--81, 1952.

\bibitem{H}
Philip Hall.
\newblock On representatives of subsets.
\newblock {\em Journal of the London Mathematical Society}, s1-10(1):26--30, 1935.

\bibitem{Ho}
Wassily Hoeffding.
\newblock Probability inequalities for sums of bounded random variables.
\newblock {\em Journal of the American Statistical Association}, 58:13--30, 1963.

\bibitem{JLR}
Svante Janson, Tomasz {\L}uczak, and Andrzej Ruci{\'n}ski.
\newblock {\em Random Graphs}.
\newblock Wiley-Interscience, 2000.

\bibitem{KSS1}
J\'anos Koml\'os, G\'abor~N. S\'ark\"ozy, and Endre Szemer\'edi.
\newblock Proof of a packing conjecture of {B}ollob\'as.
\newblock {\em Combinatorics, Probability and Computing}, 4(3):241--255, 1995.

\bibitem{KSS2}
J\'anos Koml\'os, G\'abor~N. S\'ark\"ozy, and Endre Szemer\'edi.
\newblock Spanning trees in dense graphs.
\newblock {\em Combinatorics, Probability and Computing}, 10(5):397--416, 2001.

\bibitem{K}
Michael Krivelevich.
\newblock Embedding spanning trees in random graphs.
\newblock {\em SIAM Journal on Discrete Mathematics}, 24(4):1495--1500, 2010.

\bibitem{KO}
Daniela K\"uhn and Deryk Osthus.
\newblock Embedding large subgraphs into dense graphs.
\newblock In {\em Surveys in {C}ombinatorics 2009}, pages 137--167. Cambridge University Press, 2009.

\bibitem{M}
Richard Montgomery.
\newblock Spanning trees in random graphs.
\newblock {\em Advances in Mathematics}, 356:106793, 2019.

\bibitem{MPY}
Richard Montgomery, Mat{\'\i}as Pavez-Sign{\'e}, and Jun Yan.
\newblock Ramsey numbers of trees.
\newblock {\em arXiv:2509.07934}, 2025.

\bibitem{W}
Nicholas~C. Wormald.
\newblock The differential equation method for random graph processes and greedy algorithms.
\newblock In {\em Lectures on Approximation and Randomized Algorithms}, pages 73--155. Polish Scientific Publishers, 1999.

\end{thebibliography}

\end{document}